\documentclass[10pt,reqno]{amsart}
\usepackage{amsmath}
\usepackage{amsfonts}
\usepackage{indentfirst}
\usepackage{latexsym}
\usepackage{color}
\allowdisplaybreaks
\usepackage{amsthm}
\usepackage{amssymb}
\usepackage{enumerate}
\usepackage{ulem}
\allowdisplaybreaks

\newtheorem{theorem}{Theorem}[section]
\newtheorem{lemma}[theorem]{Lemma}
\newtheorem{thm}[theorem]{Theorem}

\newtheorem{defn}[theorem]{Definition}
\newtheorem{rmk}[theorem]{Remark}
\newtheorem{exam}[theorem]{Example}
\makeatletter

\newcommand{\Rmnum}[1]{\expandafter\@slowromancap\romannumeral #1@}
\makeatother

\begin{document}
	\title
	{\bf Gromov-Hausdorff convergence theory of surfaces}
	\author
	{Jianxin Sun, Jie Zhou}
	\address{
		\newline
		Jianxing Sun:
		Academy of Mathematics and Systems Science, CAS,
		Beijing 100190, P.R. China.
		{\tt Email:sunjianxin201008@126.com}
		\newline
		\newline
		Jie Zhou:
		Academy of Mathematics and Systems Science, CAS,
		Beijing 100190, P.R. China.
		{\tt Email:zhoujie2014@mails.ucas.ac.cn}}
	\date{}
	\maketitle
\begin{abstract}
   In this paper, we use the viewpoint of Gromov-Hausdorff convergence to give some new comprehension of well known theorem,it is Huber's classification theorem\cite{Huber}\cite{MS}for  complete Riemannian surfaces immersed in $\mathbb{R}^n$ with finite total curvature( $\int_{\Sigma}|A|^2<+\infty$)  
   it depend heavily on  M\"{u}ller and  \v{S}ver\'{a}k's Hardy-estimate\cite{MS} for the curvature form of surfaces immersed in $\mathbb{R}^n$ with finite total curvature.
\end{abstract}
\section{Introduction}
Let $F:\Sigma\to \mathbb{R}^n$ be an immersion of the surface $\Sigma$ in $\mathbb{R}^n$. The total curvature of $F$ is defined by
$$\int_{\Sigma}|A|^2d\mu_g,$$
where $A$ is the second fundamental form and $g=dF\otimes dF$ is the induced metric. There are many results about the $L^p(p\ge 2)$ norm of the second fundamental form.

In general, just as it is done in \cite{L}, for extrinsic surface,  the standard viewpoint is to regard the surface as the graph over its tangent space, and the convergence is in the meaning of graph.In\cite{L}, Langer first proved surfaces with $\|A\|_{L^p}\le C (p>2)$ are locally $C^{1}$-graphs over small balls(but with uniform radius) of the tangent spaces and deduced the compactness of such surfaces in the sense of graphical convergence. Recently, Breuning considered a higher dimensional generalization of Langer's theorem in \cite{L}. He proved that an immersion $f:M^n\to \mathbb{R}^{n+l}$  with bounded volume and $\int_M|A|^pd\mu_g\le C(p>n)$ is a $C^{1,\alpha}$-graph in a uniformly small ball for $\alpha<1-\frac{n}{p}$. This is a geometric analogue of the Sobolev embedding $W^{2,p}\to C^{1,1-\frac{n}{p}}$ if we regard the second fundamental form as the ``second derivative" of an immersing mapping of submanifold. 

But in the critical case $p=2$, when  taking a glimpse at the state of Leon Simon's decomposition theorem(Lemma 2.1 in \cite{LS}), one may not think that it is easy to establish a similar graphical theorem as Langer did  in \cite{L}. In this critical case $p=n=2$, Leon Simon proved a decomposition theorem \cite[lemma 2.1]{LS} which says a surface with bounded volume and sufficiently small total curvature is an almost flat Lipschitz graph outside of some small topological disks. Using this and noticing the total curvature is equal to the Willmore functional $\int_{\Sigma}|H|^2d\mu_g$ up to a constant for a closed surface with fixed topology, he got the existence of surfaces minimizing the Willmore functional.

Another important observation of immersing surfaces with finite total curvature is the compensated compactness phenomenon obtained in \cite{MS}. In general, the total curvature only controls the $L^1$-norm of the Gaussian curvature $\int_{\Sigma}|K|d\mu_g$. But under the condition
$$\int|A|^2\le 4\pi\varepsilon,$$
M\"{u}ller and  \v{S}ver\'{a}k estimated the Hardy norm of $*Kd\mu_g$  and solved the equation
$$-\triangle v=*Kd\mu_g$$
such that $v\in L^{\infty}$.
Using this $L^{\infty}$-estimation of the metric, E.Kuwert, R. Sch\"{a}tzle, Y.X. Li \cite{KS12}\cite{KL12} and T.Riviššre \cite{Ri13} proved the compactness theorem of immersing maps $f:\Sigma_g\to \mathbb{R}^n$ with Willmore functional value
$$Will(f)< 8\pi.$$
With the compactness theorem, they gave an alternate approach of existence of the Willmore minimizer, see \cite{KS13}\cite{Ri14}. Their mainly observation is that the Willmore functional value will jump over the gap $8\pi$ as the complex structure diverges to the boundary of the moduli space $\mathcal{M}_g$. And once the complex structure $\phi_k:\Sigma_k\to \Sigma_g$ converges, they can regard $f_k\circ \phi_k$ to be conformal and consider the convergence of these mappings in weak $W^{2,2}_{loc}(\Sigma\backslash S,\mathbb{R}^n)$ topology.

Geometrically, the divergency of the complex structure means the collapsing of geodesics, which changes the local topology(or area density) and contributes a gap to the Willmore functional value. To use a geometric quantity to replace the convergence of the complex structure to rule out the collapsing phenomenon, we  cite the width of $(\Omega_a, g)$ observed by J.Y.Chern and Y.X.Li in \cite{Jingyi Chern Yuxiang Li}.

  For example, for Huber's classification theorem of the complex structure of complete immersed surface in $\mathbb{R}^n$ with finite total curvature, it is well known that the diffeomorphic structure of such surface is of the type $\Sigma_g\backslash \{p_1,p_2,...p_l\}$. So the local complex structure of $\Sigma$ around each $p_i$ is either parabolic or hyperbolic, i.e. $B^{\Sigma}_{r_i}(p_i)\backslash \{p_i\}$ is conformal to $\Omega_a=\{z\in \mathbb{C}|a<|z|\le 1\}$ for some $a\ge 0$. The goal is to exclude the hyperbolic case $a>0$.
  
  The starting point is the lower bound of the width of $(\Omega_a, g)$ observed by J.Y.Chern and Y.X.Li in \cite{Jingyi Chern Yuxiang Li}. They use M\"{u}ller and  \v{S}ver\'{a}k's Hardy-estimate to reduce the extrinsic condition $\int_{\Sigma}|A|^2<+\infty$ to the intrinsic condition $l_0(g)>0$ (positive width), which intuitively means the positivity of the injective radius of the complete metric $(\Omega_a, g)$  near the infinite boundary(lemma \ref{positive injective radius}). So the viewpoint of Gromov-Hausdorff convergence works. And we will prove, when looked from infinite(in the meaning of Gromov-Hausdorff convergence), the annulus $\Omega_a$ will be a cylinder, which means the capacity of $\Omega_a$ will vanish and contradicts to the hyperbolic assumption $a>0$. To use Gromov-Hausdorff convergence theory, we first consider the (intrinsic)flat case in section \ref{flat case} and then,in section \ref{general case}, use the $L^{\infty}$-estimate of new radius in \cite{Sun-Zhou} to reduce the general (extrinsic)case to the flat case.

\section{Preliminaries}
\subsection{Hardy estimate of Gauss curvature}
   We begin with M\"{u}ller and  \v{S}ver\'{a}k's Hardy-estimate\cite{MS} for the curvature form of a surface $\Sigma$ immersed in $\mathbb{R}^n$ with total curvature $\int_{\Sigma}{|A|}^2d\mu_g$ small. Roughly speaking, for $F:\Sigma\to\mathbb{R}^n$ an immersed Riemannian surface equipped with the induced metric, they consider the Gauss map $G:\Sigma \to \mathbb{CP}^{n-1}$, locally defined by $G(p)=[\frac{e_1(p)+\sqrt{-1}e_2(p)}{2}]$, where $\{e_1(p),e_2(p)\}$ are orthonormal basis of $\Sigma$ at the point p. When noting that the K\"{a}ller form $\omega$ over the complex projection space $\mathbb{CP}^{n-1}$, which has the algebraic structure of determinant when transgressed to the total space $S^{2n-1}$ of the Hopf fibration, is the \lq\lq classification form\rq\rq of Gauss curvature, i.e.
   $$G^{*}\omega=Kd\mu_g=K\omega^1\wedge\omega^2,$$
   where $Kd\mu_g$ is the Gauss curvature form of $\Sigma$, they can estimate the Hardy norm of the curvature form using the results of Coifman, Meyer, Lions and Semmes\cite{CLMS}(see also \cite{M}). Moreover, when noting that in dimension two, the fundamental solution belongs to $BMO$, the dual of Hardy space (Fefferman and Stein \cite{FS}), one can solve the Dirichlet problem with the curvature form as density over the whole complex plane $\mathbb{C}$. More precisely, they have the following theorem:
   \begin{thm}\label{MSFS}
      For $0<\varepsilon<1$. Assume $\varphi\in W_0^{1,2}(\mathbb{C},\mathbb{CP}^{n})$ satisfies  $\int_{\mathbb{C}}\varphi^*\omega=0$ and that $\int_{\mathbb{C}}|D\varphi\wedge D\varphi|\le2\pi\varepsilon$. Then $*\varphi^*\omega\in\mathcal{H}^1(\mathbb{C})$ with $\|\varphi^*\omega\|_{\mathcal{H}^1}\le c_1C(n,\varepsilon)\|D\varphi\|_{L^2}^2$, where $C(n,\varepsilon)=1+\frac{4n^2(1-{\varepsilon}^{\frac{1}{n}})}{(1-\varepsilon)^2}$. Moreover, the equation $\triangle u=*\varphi^*\omega$ admits a unique solution $u_0:\mathbb{C}\to\mathbb{R}$ which is continuous and satisfies:
     $$\mathop{\mathrm{lim}}_{z\to \infty}u_0(z)=0,$$
     and
     $$\int_{\mathbb{C}}|D^2u_0|+\{{\int_{\mathbb{C}}{|Du_0|}^2}\}^{\frac{1}{2}}+\mathop{\mathrm{max}}_{z\in \mathbb{C}}|u_0|(z)\le c_2\|\varphi^*\omega\|_{\mathcal{H}^1} \le c_3 C(n,\varepsilon)\|D\varphi\|_{L^2}^2,$$
     where $c_1, c_2$ and $c_3$ are constants independent of $n$ and $\varepsilon$, which will be denoted as a same notation $c$ in the following text.
   \end{thm}

   \subsection{Capacity}\label{section:Capacity}
   \begin{defn}[Capacity]
  Assume $\Omega$ is a domain in $\mathbb{C}$
  \ with smooth boundary $\Gamma_1 \ and\ \Gamma_2$ and $g$ is a Riemannian metric on $\Omega$. Then the \textit{capacity} of $(\Omega,\Gamma_1,\Gamma_2)$ for the conformal class $[g]$ is defined by $$Cap_{[g]}(\Omega)=\mathop{\textrm{inf}}_{u|_{\Gamma_1}=1,\and u|_{\Gamma_2}=0}\{\int_{\Omega}|\nabla_{g}u|^2d\mu_g|\ u\in W^{1,2}(\Omega)\}$$\\
  \end{defn}
  \begin{rmk}
   Capacity does not depend on the choice of metrics in a fixed conformal class. So, we will  always take the canonical metric $g_0=|dz|^2$ on $\Omega$ and define $Cap(\Omega)=Cap_{[g_0]}(\Omega)$.
  \end{rmk}
  \begin{lemma}
     The capacity $Cap(\Omega)$ $($i.e. the minimum of the functional $I(u)=\int_{\Omega}{|\nabla u|}^2$ on $X=\{u\in W^{1,2}(\Omega)| u=\chi_{\Gamma_1}\  \mathrm{on}\ \partial\Omega=\Gamma_1\cup \Gamma_2\}$ $)$ is achieved by the unique harmonic function which solve the Dirichlet problem $\triangle u=0$ in $\Omega$, $u=1$ on $\Gamma_1$ and $u=0$ on $\Gamma_2$.
\end{lemma}
\begin{proof}
By classical variational argument.
 \end{proof}

\begin{exam}
  Define $\Omega_a=\overline{D}\setminus\overline{D}_a$,$a\in (0,1)$ where $\overline{D}_a=\{
  z\in\mathbb{C}|0\leq z\leq a\}$ and $\overline{D}=\{z\in\mathbb{C}|0\leq z\leq 1\}$. Then $Cap(\Omega_a)=-2\pi \log a$.
\end{exam}
\begin{proof}
  W.L.O.G.,choose the canonical metric $g=|dz|^2$ on $\Omega_a$ and solve the the $Laplace$ equation $\triangle u=0$ in $\Omega_a$, $u=0$ on $\Gamma_2=\{z||z|=1\}$, $u=0$ on $\Gamma_1=\Gamma_a=\{z||z|=a\}$
    we get$$
    u_0(x,y)=\frac{1}{\log a}\log(\sqrt{x^2+y^2}),\forall (x,y)\in \Omega_a$$
    hence $Cap(\Omega_a)=\int_{a\leq |z|\leq 1}|\nabla u_0|^2dx\wedge dy=-2\pi /\log a>0.$
\end{proof}

 \section{Convergence view of $Huber$'s theorem }
  \subsection{Flat case}\label{flat case}
   In this section, we define $\Omega_{(a,b]}=\{z\in \mathbb{C}|0<a<|z|\le b\}$, $\Omega_a=\Omega_{(a,1]}$ and use
   $$
     H:=\{\gamma:\rm{S^1} \rightarrow \Omega_a | \gamma \; piecewise \; smooth \; with \; [\gamma]\neq 0\}
  $$
  to denote the set of piecewise smooth homotopic nontrivial closed curves in $\Omega_a$.
  Assume $g$ is a Riemannian metric on $\Omega_a$, we define the width of $(\Omega_a, g)$ by
  $$l_0(g):=\inf_{\gamma\in H}L_g(\gamma),$$
  where $L_g$ is the length functional.
 The goal is to prove the following intrinsic theorem of the nonexistence of flat complete metric with positive width on an annulus.
    \begin{thm} \label{nonexistence of flat complete metric}
    Any flat metric with positive width on the annulus $\Omega_a$ can not be complete.
    \end{thm}
   The whole plan is a contradiction argument, so we begin with the assumption that $g=e^{2u}|dz|^2$ is a flat complete metric on $\Omega_a$ with positive width.
  \begin{lemma}\label{lemma 2.1}
    If $g$ is a complete metric on $\Omega_a$, then the volume of $\Omega_a$ under the metric $g$ is $\mu_g(\Omega_a)=+\infty$.
  \end{lemma}
  \begin{proof}
    If not, i.e. $\mu_g(\Omega_a)<+\infty$. Take a homotopic nontrivial simple closed curve $\gamma :\rm{S^1}\rightarrow \Omega_a$ and let
    $$
    \eta(x)=
    \begin{cases}
      d_g(x,\gamma),&x \;\textrm{between}\; \Gamma_a\; \textrm{and}\; \gamma(\textrm{denote as}\; x \in \Omega_a^{\gamma}),\\
      0,&x\;\textrm{between}\; \gamma\; \textrm{and}\; \Gamma_1(\textrm{denote as}\; x \notin \Omega_a^{\gamma}).\\
    \end{cases}
    $$
especially, when choose $\gamma$ as an embedded closed curve in $\Omega$,
 then $\eta$ is a distance function in $\Omega_a^{\gamma}$, i.e. ${|\nabla_g\eta|}_g=1$ a.e. in $\Omega_a^{\gamma}$.

    Notice that for a cut function $\beta \in C^{\infty}(\mathbb{R})$ s.t.
    $$
    \beta=
    \begin{cases}
      0,& \textrm{on}\;(-\infty,\varepsilon),\\
      1,&\textrm{on}\;(A-\varepsilon,+\infty).
    \end{cases}
    $$
    and
    $$|\beta^\prime(t)|\leq \frac{2}{A}$$
    let $\eta^A=\beta\circ\eta$. Then
    $$
    |\nabla_g\eta^A(x)|=|\beta^\prime(\eta(x))||(\nabla_g\eta)(x)|\leq \frac{2}{A}.
    $$
    Observe $\eta^A|_{\Gamma_1}=0$ and $g$ complete guarantees $\mathop{\rm limit}_{x \to \Gamma_a}{\eta(x)}=+\infty$ which implies $\eta^A|_{\Gamma_a}=1$ for any fixed $A\in \mathbb{R}$.By the definition of capacity we know
    \begin{align*}
    Cap(\Omega_a)\leq \int_{\Omega_a}{|\nabla_g\eta^A|}^2d\mu_g\leq \int_{\Omega_a}{|\frac{2}{A}|}^2d\mu_g=\frac{4}{A^2}\mu_g(\Omega_a)\rightarrow 0
  \end{align*}
\noindent as $A\rightarrow+\infty$ since $\mu_(\Omega_a)<+\infty$.  So we get $Cap(\Omega_a)=0$ which contradicts to the fact $Cap(\Omega_a)=-2\pi /\log a$.
  \end{proof}
  Now, consider the conformal metric $\tilde{g}(z)=g(z){|z|}^{2\lambda}$, then $\tilde{g}$ is also a flat complete metric on $\Omega_a$ since ${|z|}^{2\lambda}\ge{|a|}^{2\lambda}$ and $\lambda\log|z|$ is harmonic in $\Omega_a$.
Furthermore, let $L_{\tilde{g}}(\gamma)=\int_0^1{|\dot{\gamma}(t)|}_{\tilde{g}}\mathrm{d}t$ be the length functional corresponding to the metric $\tilde{g}$, then
  $$
  \tilde{l}_0:=\mathop{\rm inf}\limits_{\gamma\in H}L_{\tilde{g}}(\gamma)\ge {|a|}^{2\lambda}\mathop{\rm inf}\limits_{\gamma\in H}L_{g}(\gamma)={|a|}^{2\lambda}l_0>0
  $$
  and
  \begin{align*}
    \tilde{l}_{1-\varepsilon}:=L_{\tilde{g}}(\Gamma_{1-\varepsilon})
    =\int_0^{2\pi}{(1-\varepsilon)}^{2\lambda}{|\dot{\gamma}|}_{g}\mathrm{d}t
    ={(1-\varepsilon)}^{2\lambda}l_{1-\varepsilon}
  \end{align*}
  where $l_{1-\varepsilon}:=L_g(\Gamma_{1-\varepsilon})$. So we have
  \begin{equation*}
    0<\tilde{l}_0<{(1-\varepsilon)}^{2\lambda}l_{1-\varepsilon} \tag{$\ast$}
  \end{equation*}

  Now, take a minimal sequence $(\gamma_k,p_k)$ of $L_{\tilde{g}}$ such that
  \begin{equation*}
    \left\{
    \begin{aligned}
      \left[\gamma_k\right] &\neq 0, p_k\in \mathrm{Im}\gamma_k\\
      \lim\limits_{k \to \infty} L_{\tilde{g}}(\gamma_k)&=\mathop{\rm inf}\limits_{\gamma\in H}L_{\tilde{g}}(\gamma)=\tilde{l}_0
    \end{aligned}
    \right.
  \end{equation*}
  and we get the following two lemmas.
  \begin{lemma}
    Take notations as above, then $\exists M$ large enough such that for $\lambda>M$,
    \begin{equation*}
      p_k \nrightarrow \Gamma_1
    \end{equation*}
  \end{lemma}
  \begin{proof}
    If not, assume $p_k\rightarrow\Gamma_1$, then for $k$ large enough, $p_k\in\Omega_{1-\frac{1}{2}\varepsilon}$.

    If for infinite $k$, $\mathrm{Im}\gamma_k\subseteq\bar{\Omega}_{1-\frac{1}{2}\varepsilon}$, then for these $k$
    \begin{align*}
      l(\gamma_k)=\int_0^{2\pi}{|z|}^{2\lambda}{|\dot{\gamma}_k|}_g\mathrm{d}t
                 \ge{\left(1-\tfrac{1}{2}\varepsilon\right)}^{2\lambda}\int_0^{2\pi}{|\dot{\gamma}_k|}_g\mathrm{d}t
                 \ge{\left(1-\tfrac{1}{2}\varepsilon\right)}^{2\lambda}l_0
    \end{align*}
    let $k\to +\infty$ we get
    \begin{equation*}
      \tilde{l}_0=\lim\limits_{k\to\infty}l(\gamma_k)\ge{\left(1-\tfrac{1}{2}\varepsilon\right)}^{2\lambda}l_0
    \end{equation*}\\
    so
    \begin{align*}
      (\ast)&\Longrightarrow{\left(1-\tfrac{1}{2}\varepsilon\right)}^{2\lambda}l_0\le{(1-\varepsilon)}^{2\lambda}l_{1-\varepsilon}\\
            &\Longrightarrow\lambda\le M_{\varepsilon}:=\frac{\ln (l_{1-\varepsilon}/ l_0)}{2\ln((1-\tfrac{1}{2}\varepsilon)/(1-\varepsilon))}\text{($>0$ for $\varepsilon$ small)}
    \end{align*}

    Otherwise, for infinite $k$, $\mathrm{Im}\gamma_k\cap\Gamma_{1-\frac{1}{2}\varepsilon}\ne\varnothing$, but
    $\gamma_k(0)=p_k\in\Omega_{1-\frac{1}{2}\varepsilon}$, so there exists a first $t_k\in [0,2\pi]$ such that $
    q_k=\gamma_k(t_k)\in\Gamma_{1-\frac{1}{2}\varepsilon}$ and hence
    \begin{align*}
      L_{\tilde{g}}(\gamma_k)&\ge\int_0^{t_k}{|\dot{\gamma}_k|}_{\tilde{g}}\mathrm{d}t
                 \ge\int_0^{t_k}{(1-\tfrac{1}{2}\varepsilon)}^{2\lambda}{|\dot{\gamma}_k|}_g\mathrm{d}t
                 \ge{(1-\tfrac{1}{2}\varepsilon)}^{2\lambda} \mathrm{d}(p_k,q_k)
    \end{align*}
    let $k\to +\infty$£¬and note $p_k\rightarrow\Gamma_1$, we get
    \begin{equation*}
      \tilde{l}_0\ge{(1-\tfrac{1}{2}\varepsilon)}^{2\lambda}\mathrm{d}_{1-\frac{1}{2}\varepsilon}
    \end{equation*}
    where $\mathrm{d}_{1-\frac{1}{2}\varepsilon}:=\mathrm{d}_g(\Gamma_{1-\frac{1}{2}\varepsilon},\Gamma_1)$. Again by $(\ast)$ we get
    \begin{equation*}
      \lambda\le N_{\varepsilon}:=\frac{\ln (l_{1-\varepsilon}/ d_{1-\frac{1}{2}\varepsilon})}{2\ln((1-\tfrac{1}{2}\varepsilon)/(1-\varepsilon))}\text{($>0$ for $\varepsilon$ small)}
    \end{equation*}

    So if we take $M=M_{\varepsilon}+N_{\varepsilon}+1$ which is independent on $\lambda$, then for $\lambda\ge M$,
    \begin{equation*}
      p_k\nrightarrow\Gamma_1
    \end{equation*}
  \end{proof}
  \begin{lemma}\label{realize injective radius}
    Take notations as above, and assume $\forall q_k\rightarrow \Gamma_a$, $\mathop{\mathrm {inj}}_{\tilde{g}}(q_k)\rightarrow+\infty$, then $\exists M$ large enough such that for $\lambda>M$,
    \begin{equation*}
      p_k \nrightarrow \Gamma_a
    \end{equation*}
  \end{lemma}
  \begin{proof}
   If not, $p_k \rightarrow \Gamma_a$, then $\mathrm{d}_{\tilde{g}}(p_k,\Gamma_1)\to +\infty$ since $\tilde{g}$ is complete, so we could assume $p_k\in D_{(a,1-\varepsilon)}=\{z\in \Omega|a<|z|<1-\varepsilon\}$ with out loss of generality.

   Claim: under the condition $\lambda\ge M=N_{\varepsilon}$, we have $\mathop{\mathrm{inj}}_{\tilde{g}}(p_k)$ is less than $\mathrm{d}_{\tilde{g}}(p_k,\Gamma_1)$ and hence is realized by a geodesic loop $\tilde{\gamma}_k$ with $L_{\tilde{g}}(\tilde{\gamma}_k)\le L_{\tilde{g}}(\gamma_k)$\, (as a result, $L_{\tilde{g}}(\tilde{\gamma}_k)\to \tilde{l}_0$).

   In fact, it is enough to prove that for any $\gamma\in H$ with $\gamma(0)=p_k$, if $\exists t_1\in [0,2\pi]$ such that $\gamma(t_1)\in\Gamma_1$, then there exists a shorter $\tilde{\gamma}\in H$ such that $\tilde{\gamma}(0)=p_k$ and $\mathrm{Im}\tilde{\gamma}\cap\Gamma_1=\varnothing$. To find such $\tilde{\gamma}$, we note
   \begin{equation*}
     \left\{
     \begin{aligned}
     &\gamma(0)=\gamma(1)=p_k\in D_{(a,1-\varepsilon)}\\
     &\gamma(t_1)\in\Gamma_1
     \end{aligned}
     \right.
   \end{equation*}
   $\Longrightarrow$ there exist the minimal $t_2\in[0,t_1]$ and then the maximal $t_3\in[0,t_2]$ such that $\gamma(t_2)\in\Gamma_{1-\frac{1}{2}\varepsilon}$ and $A:=\gamma(t_3)\in \Gamma_{1-\varepsilon}$, the maximal $t'_2\in[t_1,2\pi]$ and then the minimal $t'_3\in[t'_2,2\pi]$ such that $\gamma(t'_2)\in\Gamma_{1-\frac{1}{2}\varepsilon}$ and $B:=\gamma(t'_3)\in\Gamma_{1-\varepsilon}$.
   let
   \begin{equation*}
     \tilde{\gamma}(t)=
     \begin{cases}
       \gamma(t), t\in [0,t_3]\cup[t'_3,2\pi]\\
       \widetilde{AB}(\mathrm{the\ clockwise\ arc\ on\ the\ circle\ \Gamma_{\varepsilon}\ joining\ A\ and\ B }), t\in[t_3,t'_3]
     \end{cases}
   \end{equation*}
   then $\mathrm{Im}\tilde{\gamma}\cap\Gamma_1=\varnothing$, $\tilde{\gamma}\in H$ (since $[p_k\widetilde{AB}p_k]+[p_k\widetilde{BA}p_k]=[ABA]=[S^1]\ne 0$, we may change $\widetilde{AB}$ to $\widetilde{BA}$ if necessary to guarantee $[\tilde{\gamma}]\ne 0$) and
   \begin{align*}
   L_{\tilde{g}}(\gamma)-L_{\tilde{g}}(\tilde{\gamma})
         \ge\tilde{\mathrm{d}}(A,\gamma(t_1))-L_{\tilde{g}}(\Gamma_{1-\varepsilon})
         \ge\tilde{\mathrm{d}}(\Gamma_{1-\frac{1}{2}\varepsilon},\Gamma_1)-L_{\tilde{g}}(\Gamma_{1-\varepsilon})
         >0
   \end{align*}
   where we use the condition $\lambda>N_{\varepsilon}$ in the last step. In fact,
   if $\tilde{\mathrm{d}}(\Gamma_{1-\frac{1}{2}\varepsilon},\Gamma_1)-L_{\tilde{g}}(\Gamma_{1-\varepsilon})\le 0$, then $$(1-\varepsilon)^{2\lambda}l_{1-\varepsilon}\ge\tilde{l}_{1-\varepsilon}=L_{\tilde{g}}(\Gamma_{1-\varepsilon})
   \ge\tilde{d}(\Gamma_{1-\frac{1}{2}\varepsilon},\Gamma_1)
   \ge(1-\frac{1}{2}\varepsilon)^{2\lambda}d_{1-\frac{1}{2}\varepsilon},$$
   i.e. $\lambda\le N_{\varepsilon}$.
   we find the $\tilde{\gamma}$ and prove the claim.

   Now, minimize the pointed class $H_k:=\{[\gamma]\ne 0|\gamma(0)=\gamma(1)=p_k\}$ and we get a geodesic loop $\tilde{\gamma}_k$ which realizes $2\mathop{\mathrm{inj}}_{\tilde{g}}(p_k)=L_{\tilde{g}}(\tilde{\gamma}_k)$. So, $\mathop{\mathrm{inj}}_{\tilde{g}}(p_k)\to\frac{1}{2}\tilde{l}_0<+\infty$ as $k\to\infty$---A contradiction!
  \end{proof}
  \begin{rmk}
    During the proof, we see that $\mathop{\mathrm{inj}}_{\tilde{g}}(p)$ could be realized by a homotopic nontrivial geodesic loop if $\lambda>N_{\varepsilon}$ and $p\in D_{(a,1-\varepsilon)}$ although $(\Omega_a,\tilde{g})$ is a complete manifold with boundary.
  \end{rmk}

  The following conclusion of the two lemmas above means, somehow, the complex structure of a complete annulus with positive width could control its injective radius.
  \begin{lemma}\label{positive injective radius}
     Assume $g$ is a fixed flat complete metric with positive width on $\Omega_a$, then the modified metric $\tilde{g}(z)={|z|}^{2\lambda}g(z)$ is also flat and complete. Moreover, there  exists a constant $M>0$, such that for any $\lambda\ge M$, one of the following happens:
    \begin{enumerate}[\rm(a)]
      \item  $\forall q_k \to \Gamma_a$, $\mathop{\mathrm {inj}}_{\tilde{g}}(q_k)\rightarrow+\infty$. In this case, for any mimimal sequence $\gamma_k$ of $L_{\tilde{g}}$, there exists a compact set $K\subset\subset\Omega$ with $K\cap\Gamma_1=\varnothing$ such that $\mathrm{Im}\gamma_k\subset K$ for $k$ large enough;
      \item $\exists p_k\rightarrow \Gamma_a$, $\mathop{\mathrm {inj}}_{\tilde{g}}(p_k)\rightarrow\frac{1}{2}\tilde{l}_0$.
    \end{enumerate}
  \end{lemma}
  \begin{proof}
    Add all the lemmas above together.
  \end{proof}

  \begin{proof}[Proof of Theorem~\ref{nonexistence of flat complete metric}]
    Assume there is a flat complete metric $g=df\otimes df=e^{2u}|dz|^2$ for some conformal immersion $f\in W^{2,2}_{conf,loc}(\Omega_a)$ on $\Omega_a$ and take notations as before.

    On the one hand, if case $\rm(a)$ happens  in the last lemma, then to minimizing $L_{\tilde{g}}$ in $H$ is equal to minimize it in the subset $H_K=\{\gamma\in H|\mathrm{Im}\gamma\subset K\}$. But then $Cartan$'s existence theorem of closed geodesics\cite[sec. 12.2]{do Carmo} implicates that there exists a closed geodesic $\gamma:\rm{S^1}\to K\subset\subset \Omega$ (investigate the proof of $Cartan$'s theorem and note the compactness of $K$ works when using $Arzela-Ascoli$ Lemma). Now, as in Lemma~\ref{lemma 2.1}, we let
    $$
    \eta(x)=
    \begin{cases}
      d_{\tilde{g}}(x,\gamma),&x \;\textrm{between}\; \Gamma_a\; \textrm{and}\; \gamma(\textrm{denote as}\; x \in \Omega_{\gamma}),\\
      0,&x\;\textrm{between}\; \gamma\; \textrm{and}\; \Gamma_1(\textrm{denote as}\; x \notin \Omega_{\gamma}).\\
    \end{cases}
    $$
    and $\beta\text{, }\eta^A$ as before, then by using Fermi's coordinate(see the second remark below) corresponding to the normal exponential map of the geodesic submanifold $\gamma$,  we get
    \begin{align*}
      Cap(\Omega_a)\le\int_{\Omega_a}{|\nabla_g\eta^A|}^2\mathrm{d}\mu_g
                 \le\int_{\varepsilon}^{A-\varepsilon}\int_0^{2\pi}{|\frac{2}{A}|}^2C\mathrm{d}s\mathrm{d}t
                 =\frac{4 C(A-2\varepsilon)L_{\tilde{g}}(\gamma)}{A^2}.
    \end{align*}
     Let $A\to +\infty$, then we have $Cap(\Omega_a)=0$, a contradiction!

    On the other hand, if case $(\rm{b})$ happens in the last lemma, consider the sequence of flat complete pointed manifolds $\{(\Omega_k:=\Omega_a,g_k:=\tilde{g},p_k)\}_{k=1}^{\infty}$. Take $p_k'\in\Omega_a^{p_k}:=\{z\in\mathbb{C}|a<|z|\le|p_k|\}$ such that  $\mathrm{inj}_{\tilde{g}}(p_k')=\inf_{p\in \Omega_a^{p_k}}\mathrm{inj}_{\tilde{g}}(p)$, then $p_k'\to\Gamma_a$. Furthermore, by the remark after lemma~\ref{realize injective radius}, there exist $\gamma_k\in H$ such that $\mathrm{inj}_{\tilde{g}}(p_k')=\frac{1}{2}L_{\tilde{g}}(\gamma_k)$. Then by definition of $\tilde{l}_0$, we have
    $$\frac{1}{2}\tilde{l}_0\le\frac{1}{2}\liminf_{k\to\infty}L_{\tilde{f}}(\gamma_k)
    =\liminf_{k\to\infty}\mathrm{inj}_{\tilde{g}}(p_k')\le\limsup_{k\to\infty}\mathrm{inj}_{\tilde{g}}(p_k)
    =\frac{1}{2}\tilde{l}_0,$$
    i.e. $\lim_{k\to\infty}\mathrm{inj}_{\tilde{g}}(p_k')=\frac{1}{2}\tilde{l}_0$. So, W.L.O.G. we may assume $p_k=p_k'$. Thus the sequence of pointed flat(hence Einstein) complete manifolds $\{(\Omega_k:=\Omega_a^{\frac{1}{2}p_k},g_k:=\tilde{g},\frac{1}{2}p_k)\}_{k=1}^{\infty}$ have injective radius uniformly bounded from below by $\frac{1}{4}\tilde{l}_0$ hence converges\cite[sec. 10.4, 10.5]{P}(see also $Anderson$'s original article \cite{A}) to a limit flat complete manifold(without boundary) $(\Sigma_{\infty},g_{\infty},p)$ smoothly when passing to a subsequence.
    i.e. there exist local diffeomorphisms $f_k:(\Sigma_{\infty},p)\rightarrow(\Omega^k\subset\Omega_a, p_k)$ with $f_k^{*}(g_k|_{\Omega^k})\rightarrow g_{\infty}$ smoothly on compact subsets of $\Sigma_{\infty}$. So, by the well known uniformization theorem for Riemannian surface,
    $(\Sigma_{\infty},g_{\infty},p)=\mathbb{R}^2,\mathrm{S^1}\times\mathbb{R}$ or $\mathrm{S^1}\times\mathrm{S^1}$.
    But we also have $\mathop{\mathrm{inj}}_{\tilde{g}}(p_k)\rightarrow\frac{1}{2}\tilde{l}_0<+\infty$, so there exist geodesic loops at$\{p_k\}_{k=1}^{\infty}$ such that
    \begin{equation*}
      \left\{
      \begin{aligned}
        &\gamma_k(0)=p_k\\
        &L_{\tilde{g}}(\gamma_k)\to \tilde{l}_0.\\
      \end{aligned}
      \right.
    \end{equation*}
    Let $\gamma_k^{\infty}(t):=f_k^{-1}(\gamma_k(t))$, then $\gamma_k^{\infty}(t)$ converges to a geodesic loop
    $\gamma^{\infty}(t)$ on $\Sigma_{\infty}$ with $L_{g_\infty}(\gamma^{\infty})=\tilde{l}_0<+\infty$ and so $\Sigma_{\infty}\ne\mathbb{R}^2$. It could also not be the compact manifold $\mathrm{S^1}\times\mathrm{S^1}$ since noncompact manifolds do not converges to compact manifold, so $\Sigma_{\infty}$ must be
    $\mathrm{S^1}\times\mathbb{R}$.

     Choose a closed circle(geodesic) $\mathrm{S^1}$ on $\Sigma_{\infty}$ and define
     $$
    \eta(x)=
    \begin{cases}
      d_{g_{\infty}}(x,\gamma),&x \text{ at one side of $\mathrm{S^1}$}\\
      0,&x\text{ at the other side of $\mathrm{S^1}$}.\\
    \end{cases}
    $$
    For $\forall \delta>0$, choose $\beta$, $\eta^A$ and $A$ large enough as before, we get
    $$\int_{\Sigma_{\infty}}{|\nabla_{g_\infty}\eta^A|}_{g_{\infty}}^2\mathrm{d}\mu_{\infty}<\delta$$
    Now, let $\eta_k^A:=\eta^A\circ f_k^{-1}$, then for $h_k:=f_k^*\tilde{g}$, we have
    \begin{align*}
      {|\nabla_{h_k}\eta^A|}_{h_k}={|\nabla_{\tilde{g}}\eta_k^A|}_{\tilde{g}}
    \end{align*}
    and hence
    \begin{align*}
      \int_{\Omega_a^k}{|\nabla_{\tilde{g}}\eta_k^A|}_{\tilde{g}}^2\mathrm{d}\mu_k
      \le\int_{\Sigma_{\infty}}{|\nabla_{h_k}\eta^A|}_{h_k}^2\mathrm{d}\mu_k
      \to\int_{\Sigma_{\infty}}{|\nabla_{g_{\infty}}\eta^A|}_{g_{\infty}}^2\mathrm{d}\mu_{\infty}<\delta.
    \end{align*}
     So, for $k$ large enough,
    \begin{align*}
      Cap(\Omega)\le\int_{\Omega_a^k}{|\nabla_{\tilde{g}}\eta_k^A|}_{\tilde{g}}^2\mathrm{d}\mu_k
                 \le\int_{\Sigma_{\infty}}{|\nabla_{g_{\infty}}\eta^A|}_{g_{\infty}}^2\mathrm{d}\mu_{\infty}+\delta
                 \le 2\delta.
    \end{align*}
    thus $Cap(\Omega)=0$, again a contradiction!
  \end{proof}

  \begin{rmk}
    To get the volume form under the Fermi's coordinate corresponding to $\gamma$, we should solve Jacobi's equation
    \begin{equation*}
    \left\{
    \begin{aligned}
      &J''+R(\dot{\sigma},J)\dot{\sigma}=0\\
      &J'(0)+A_{\dot{\sigma}(0)}(J'(0))\perp TN\\
      &J(0)\in TN
    \end{aligned}
    \right.
    \end{equation*}
    where $A$ is the Weingarten's transform of the submanifold $N$ and $\sigma$ is a geodesic perpendicular to $N$. In the case $N$ is a geodesic on a surface, the Weingarten's transform vanishes. So when just consider those Jacobi fields perpendicular to the chosen geodesic $\sigma$ with initial velocity $J'(0)\perp\dot{\sigma}(0)$, the equation becomes to
    \begin{equation*}
    \left\{
    \begin{aligned}
      &J''+R(\dot{\sigma},J)\dot{\sigma}=0\\
      &J'(0)=0\\
      &J(0)=C\gamma'(0)\\
    \end{aligned}
    \right.
    \end{equation*}
    whose unique solution satisfying ${|J(t)|}_{\tilde{g}}\equiv C$. In all, the volume form under the Fermi's coordinate is $vol_{\tilde{g}}=C\mathrm{d}s\mathrm{d}t$.
  \end{rmk}

  \subsection{General case}\label{general case}
  Now, it is possible to prove $Huber$'s theorem\cite{Huber}. The key point is the next lemma observed by Jingyi Chern and Yuxiang Li in \cite{Jingyi Chern Yuxiang Li}, which reduces the extrinsic condition $\int_{\Sigma}|A|^2<\infty$ to the intrinsic condition $l_0(g)>0$.
  \begin{lemma}\label{positive width}
     Assume $g=df\otimes df=e^{2u}|dz|^2$ for some conformal immersion $f\in W^{2,2}_{conf,loc}(\Omega_a)$ is a complete metric on $\Omega_a$ and consider the length functional $L_g$ on $H$, then
    $$
    l_0:=\mathop{\rm{inf}}_{\gamma \in H}L_g(\gamma)>0
    $$
  \end{lemma}
  \begin{proof}
    See \cite[prop. 3.4]{Jingyi Chern Yuxiang Li}
  \end{proof}
  \begin{thm}
  Assume $\Sigma$ is a complete surface immersed in $\mathbb{R}^n$ with finite total curvature $\int_{\Sigma}|A|^2d\mu_g<\infty$, then $\Sigma$ is conformally
equivalent to a compact Riemannian surface with finitely many points deleted.
  \end{thm}
  \begin{proof}
   By the theorem from differential topology , $\Sigma$ is diffeomorphic to $\Sigma_g\backslash \{p_1,\ p_2\ldots p_m\}$ for some compact surface $\Sigma_g$ with genus $g$. Around each $p_i$, the complex structure of $\Sigma$ is either $D_1\backslash \{0\}$ or $\Omega_a$($a>0$).  We just need to rule out the latter one.

   If it is in the latter case, then there exists a small punctured neighborhood $\mathring{B}(p_i)$ of $p_i$ in $\Sigma_g$ such that $\int_{\mathring{B}(p_i)}|A|^2d\mu_g=\varepsilon_1\ll 1$ and a conformal parametrization $f:\Omega_a\to \mathring{B}(p_i)\subset \Sigma\looparrowright\mathbb{R}^n$. So, by lemma~\ref{positive width},
   $$l_0(g)=\mathop{\rm{inf}}_{\gamma \in H}L_g(\gamma)>0.$$
   Set $g=df\otimes df=e^{2u}g_0$ and let $\varphi=G\circ f:\Omega_a\to \mathbb{CP}^{n-1}$ be the representation of Gauss map under the conformal coordinate $f$, then
   $$-\triangle u=K e^{2u}=*\varphi^*\omega\ \mathrm{in}\ \Omega_a\  \mathrm{and}\  \int_{\Omega_a}|D\varphi|^2=\int_{\mathring{B}(p_i)}|A|^2d\nu_g\ll 1.$$
   Since $\mathbb{CP}^{n-1}$ is compact, we can extend $\varphi$ to $\tilde{\varphi}\in W^{1,2}(D_1, \mathbb{CP}^{n-1})$ such that $$\int_{D_1}|D\tilde{\varphi}\wedge D\tilde{\varphi}|\le\frac{1}{4}\int_{D_1}|D\tilde{\varphi}|^2\le C(a) \int_{\Omega_a}|D\varphi|^2\le\pi\varepsilon\ll 1.$$
   As in lemma~\ref{extend}, let
   $$
    \bar{\varphi}(z)=
    \begin{cases}
    \tilde{\varphi}(z), z\in D_1,\\
    \tilde{\varphi}(\frac{1}{\bar{z}}), z\in D_1^*=\{z||z|\ge1\},
    \end{cases}
    $$
   then by theorem~\ref{MSFS}, there exists a bounded continuous function $u_0$ solving the equation $-\triangle u_0=*\bar{\varphi}^*\omega$ on $\mathbb{C}$ such that
   $\mathop{\mathrm{lim}}_{z\to \infty}u_0(z)=0$. Let $h=u-u_0$, then $g_1:=e^{2h}g_0$ is a flat complete metric on $\Omega_a$ with
   $$l_0(g_1):=\mathop{\rm{inf}}_{\gamma \in H}L_{g_1}(\gamma)\ge e^{\min{u_0}}\mathop{\rm{inf}}_{\gamma \in H}L_g(\gamma)\ge e^{-C}l_0(g)>0,$$
   since $u_0$ is bounded and $\triangle h=0$. This contradicts to theorem \ref{nonexistence of flat complete metric}.
  \end{proof}

\end{document}